\documentclass[11pt,reqno]{amsart}
\usepackage{graphicx}
\usepackage{float}
\usepackage[caption = false]{subfig}
\usepackage{enumerate}
\usepackage{mathtools}
\usepackage{breqn}
\usepackage{array, geometry, graphicx}
\usepackage{amsmath,amsfonts,amssymb,amsthm}
\usepackage{multirow}
\newtheorem{theorem}{Theorem}[section]
\newtheorem{corollary}[theorem]{Corollary}
\newtheorem{lemma}[theorem]{Lemma}

\theoremstyle{definition}
\newtheorem{definition}[theorem]{Definition}
\newtheorem{example}[theorem]{Example}
\theoremstyle{remark}
\newtheorem{remark}[theorem]{Remark}
\newtheorem{note}[theorem]{Note}
\numberwithin{equation}{section}
\DeclareMathOperator{\RE}{Re} \DeclareMathOperator{\IM}{Im}
\newcommand*{\backin}{\rotatebox[origin=c]{-180}{$\in$}}

\begin{document}

\title{Application of Pythagorean means and Differential Subordination}
\thanks{$^*$Corresponding Author\\The second author is supported by The Council of Scientific and Industrial Research(CSIR). Ref.No.:08/133(0018)/2017-EMR-I.}
\author[S. Sivaprasad Kumar]{S. Sivaprasad Kumar}
\address{Department of Applied Mathematics, Delhi Technological University, Delhi--110042, India}
\email{spkumar@dce.ac.in}
\author{Priyanka Goel$^*$}
\address{Department of Applied Mathematics, Delhi Technological University, Delhi--110042, India}
\email{priyanka.goel0707@gmail.com}


\begin{abstract}
For $0\leq\alpha\leq 1,$ let $H_{\alpha}(x,y)$ be the convex weighted harmonic mean of $x$ and $y.$ We establish differential subordination implications of the form
\begin{equation*}
	H_{\alpha}(p(z),p(z)\Theta(z)+zp'(z)\Phi(z))\prec h(z)\Rightarrow p(z)\prec h(z),
\end{equation*} where $\Phi,\;\Theta$ are analytic functions and $h$ is a univalent function satisfying some special properties. Further, we prove differential subordination implications involving a combination of three classical means. As an application, we generalize many existing results and obtain sufficient conditions for starlikeness and univalence.
\end{abstract}

\maketitle

\section{Introduction}
Let $\mathbb{D}:=\{z\in\mathbb{C}:|z|<1\}$ be the open unit disc and $\mathcal{H}$ be the class of functions, analytic on $\mathbb{D}.$ For any complex number $a$ and a positive integer $n,$ let $\mathcal{H}[a,n]$ be the subclass of $\mathcal{H}$ containing functions of the form $f(z)=a+a_nz^n+a_{n+1}z^{n+1}+\cdots.$ Functions in $\mathcal{H}$ with the normalization $f(0)=0,$ $f'(0)=1$ constitutes a subclass of $\mathcal{H},$ which we denote by $\mathcal{A}.$ Further, $\mathcal{S}$ denotes the subclass of $\mathcal{A}$ consisting of univalent functions. The class of functions $p$ of the form $p(z)=1+c_1z+c_2z^2+\cdots$ with the property $\RE p(z)>0$ on $\mathbb{D},$ is known as the Carath\'eodory class and is denoted by $\mathcal{P}.$ For any two analytic functions $f$ and $F,$ it is said that $f$ is subordinate to $F,$ written as $f\prec F$ if there exists an analytic function $\omega(z),$ satisfying $\omega(0)=0$ and $|\omega(z)|<1,$ such that $f(z)=F(\omega(z)).$ In particular, if $F$ is univalent then $f\prec F$ if and only if $f(0)=F(0)$ and $f(\mathbb{D})\subset F(\mathbb{D}).$ Let $\mathcal{R}$ denotes the class of functions $f\in\mathcal{A}$ satisfying $\RE f'(z)>0$ on $\mathbb{D},$ then clearly $\mathcal{R}\subset\mathcal{S}.$ A subclass $\mathcal{S}^*$ of $\mathcal{S},$ known as the class of starlike functions, consists of the functions $f$ satisfying $\RE(zf'(z)/f(z))>0$ in $\mathbb{D}.$ The class of convex functions, denoted by $\mathcal{C},$ is another subclass of $\mathcal{S}$ consisting of the functions $f$ with $\RE(1+zf''(z)/f'(z))>0$ in $\mathbb{D}.$ An analytic function $f$ is said to be close-to-convex if $\RE(zf'(z)/g(z))>0\;(z\in\mathbb{D}),$ where $g$ is a starlike function. The class of close-to-convex functions is denoted by $\mathcal{K}.$ Let $\mathcal{S}^*(\alpha)$ $(0\leq\alpha<1)$ be the subclass of $\mathcal{S}^*$ consisting of the functions $f$ satisfying $\RE(zf'(z)/f(z))>\alpha,$ known as the class of starlike functions of order $\alpha.$ The class of $\alpha$-convex  functions of order $\beta$ is defined as
\begin{equation*}
	\mathcal{M}_{\alpha}(\beta):=\left\{f\in\mathcal{A}:	\RE\left((1-\alpha)\dfrac{zf'(z)}{f(z)}+\alpha\left(1+\dfrac{zf''(z)}{zf'(z)}\right)\right)>\beta\right\}
\end{equation*}
and the class of $\alpha$-starlike functions of order $\beta$ is defined as
\begin{equation*}
	\mathcal{L}_{\alpha}(\beta): =\left\{f\in\mathcal{A}:\RE\left(\dfrac{zf'(z)}{f(z)}\right)^{1-\alpha}\left(1+\dfrac{zf''(z)}{f'(z)}\right)^{\alpha}>\beta\right\},	
\end{equation*}
where $z\in\mathbb{D},$ $\alpha$ is any real number and $0\leq\beta<1.$ For $0\leq\alpha\leq1,$ the three classical means namely arithmetic, geometric and harmonic mean of two numbers $x$ and $y,$ in their special form, known as convex weighted means are defined respectively as follows:
\begin{eqnarray*}
	A_{\alpha}(x,y)= (1-\alpha)x+\alpha y;\;
	G_{\alpha}(x,y)= x^{\alpha}y^{1-\alpha};\;
	H_{\alpha}(x,y)= \dfrac{xy}{\alpha y+(1-\alpha)x}.
\end{eqnarray*}
Note that the classes $\mathcal{M}_{\alpha}(\beta)$ and $\mathcal{L}_{\alpha}(\beta)$ are defined by using the convex arithmetic mean and convex geometric mean, of the quantities $zf'(z)/f(z)$ and $1+zf''(z)/f'(z)$ respectively. Let $h$ be univalent on $\mathbb{D}$ and $\psi:\mathbb{C}^2\times\mathbb{D}\rightarrow\mathbb{C}$ be a complex function, then an analytic function $p,$ which satisfies
\begin{equation}\label{h15}
	\psi(p(z),zp'(z))\prec h(z)
\end{equation}
is called the solution of the differential subordination~\eqref{h15}. The theory of differential subordination introduced by Miller and Mocanu~\cite{ds} is developed on the basis of the following implication
\begin{equation*}
	\psi(p(z),zp'(z))\prec h(z)\Rightarrow p(z)\prec q(z),\qquad z\in\mathbb{D},
\end{equation*}
where $h$ and $q$ are univalent in $\mathbb{D}$ and $q$ is a dominant of the solutions of the differential subordination~\eqref{h15}. In this direction, a good amount of work has been carried out by many authors(see~\cite{pri3,first,ssp,adiba,vibha,shagun4}). In 1996, Kanas et. al~\cite{kanas1996} introduced and studied differential subordinations involving geometric mean of $p(z)$ and $p(z)+zp'(z).$ Later in 2011, the authors in~\cite{lecko2} proved several differential subordination results associated with arithmetic as well as geometric mean of certain analytic functions. In~\cite{kanasels}, Cri\c{s}an and Kanas considered a combination of arithmetic and geometric mean for which they established the following implication:
\begin{equation}\label{h16}
	\gamma (p(z))^{\delta}+(1-\gamma)(p(z))^{\mu}\left(p(z)+\dfrac{zp'(z)}{p(z)}\right)^{1-\mu}\prec h(z)\Rightarrow p(z)\prec h(z).
\end{equation}
This result is applied to prove some univalence and starlikeness criteria. A particular form of this expression, for different choices of $h(z)$ has been worked upon by Kanas~\cite{kanasbmms}. Recently, Gavri\c{s}~\cite{gavris} melded all the three pythagorean means in one expression and proved an implication similar to~\eqref{h16} for a specific choice of $h(z).$ In the present investigation, we extend the results of~\cite{gavris} by proving a similar implication for another choice of $h(z)$ and generalize many existing results. In 2014, Cho et. al~\cite{geo} established conditions on an analytic function $\Phi(z)$ so that the geometric mean of $p(z)$ and $p(z)+zp'(z)\Phi(z)$ is subordinate to $h(z)$ leads $p(z)$ to be subordinate to $h(z),$ where $h$ is a univalent function. Later in the same year, Chojnacka and Lecko~\cite{harmonic} proved a similar result for harmonic mean. The arithmetic mean is already covered by Miller and Mocanu in~\cite{ds}. In our study, we prove some general differential subordination results involving the harmonic mean of the quantities $p(z)$ and $p(z)\Theta(z)+zp'(z)\Phi(z),$ where $\Theta(z)$ and $\Phi(z)$ are analytic functions. As an application to our main results, we establish various differential subordination results pertaining to several subclasses of $\mathcal{S}^*.$ Further, we obtain sufficient conditions for starlikeness and univalence, which generalize some earlier known results. Here are a few prerequisites for our main findings:
\begin{definition}\label{defH}
	Let $t\in[0,1]$ and $\Theta, \Phi\in\mathcal{H}$ with $\Theta(0)=1.$ By $\mathcal{H}(t; \Theta, \Phi),$ we mean the subclass of $\mathcal{H}$ of all functions $f$ such that
	\begin{equation}\label{5}
		H_{t;\Theta,\Phi,f}(z):=\begin{cases}
			\dfrac{P_{0;\Theta,\Phi,f}(z)P_{1;\Theta,\Phi,f}(z)}{P_{1-t;\Theta,\Phi,f}(z)} & \quad P_{1-t;\Theta,\Phi,f}(z)\neq0, \\
			\lim_{\mathbb{D}\backin\zeta\rightarrow z}\dfrac{P_{0;\Theta,\Phi,f}(\zeta)P_{1;\Theta,\Phi,f}(\zeta)}{P_{1-t;\Theta,\Phi,f}(\zeta)} & \quad P_{1-t;\Theta,\Phi,f}(z)=0,
		\end{cases}
	\end{equation}
	is an analytic function in $\mathbb{D},$
	where
		\begin{equation*}
		P_{t;\Theta,\Phi,f}(z):=(1-t+t\Theta(z))f(z)+t\Phi(z)zf'(z),\quad z\in\mathbb{D}
	\end{equation*}
and define $H_{t;\Theta,\Phi,0}\equiv 0.$
\end{definition}
\begin{definition}
Let $\mathcal{Q}$ denotes the class of all convex functions $h$ with the following properties:
	\begin{enumerate}
		\item $h(\mathbb{D})$ is bounded by finitely many smooth arcs which form corners at their end points (including corners at infinity),
		\item $E(h)$ is the set of all points $\zeta\in\partial\mathbb{D}$ which corresponds to corners $h(\zeta)$ of $\partial h(\mathbb{D}),$
		\item $h'(\zeta)\neq 0$ exists at every $\zeta\in\partial\mathbb{D}\backslash E(h).$
\end{enumerate}
\end{definition}
\begin{lemma}[Lemma~2.2a]\cite{ds}\label{jack}
  Let $z_0\in\mathbb{D}$ and $r_0=|z_0|.$ Let $f(z)=a_nz^n+a_{n+1}z^{n+1}+\cdots$ be continuous on $\overline{\mathbb{D}}_{r_0}$ and analytic on $\mathbb{D}\cup\{z_0\}$ with $f(z)\neq 0$ and $n\geq 1.$ If $$|f(z_0)|=\max\{|f(z)|:z\in\overline{\mathbb{D}}_{r_0}\}$$ then there exists an $m\geq n$ such that $$\dfrac{z_0f'(z_0)}{f(z_0)}=m\;\text{and}\;\RE\left(\dfrac{z_0f''(z_0)}{f'(z_0)}+1\right)\geq m.$$
\end{lemma}

\section{Main Results}
We begin with the following result.
  \begin{lemma}\label{mainn}
	Let $\delta\in[0,1],$ $h\in\mathcal{Q}$ with $0\in\overline{h(\mathbb{D})}$ and $\Theta,\Phi\in\mathcal{H}$ be such that $\Theta(0)=1,$ $\RE \Phi(z)>0\;(z\in\mathbb{D})$ and
	\begin{equation}\label{hypo}
		\RE\left( \Phi(z)+\dfrac{h(\zeta)}{\zeta h'(\zeta)}(\Theta(z)-1)\right)>0,\quad z\in\mathbb{D},\;\zeta\in\partial\mathbb{D}.
	\end{equation}
	If $p\in\mathcal{H}(\delta;\Theta,\Phi),$ $p(0)=h(0)$ and $H_{\delta;\Theta,\Phi,p}\prec h,$ then $p\prec h.$
\end{lemma}
\begin{proof}
	From~\eqref{5} it is easy to see that $H_{0;\Theta,\Phi,p}= p,$ therefore the implication given in the hypothesis holds when $\delta=0$. Now let us take $\delta\in(0,1].$ If $p= p(0)\in\mathcal{H}(\delta;\Theta,\Phi),$ then using the fact that $\Theta(0)=1$ and from~\eqref{5}, we have $p(0)\in h(\mathbb{D})$ and thus in this case, the implication holds obviously. Now let $p\in\mathcal{H}(\delta;\Theta,\Phi)$ be a nonconstant function and define
	\begin{equation}\label{6}
		x:=p(z_0)\quad \text{and}\quad y:=p(z_0)\Theta(z_0)+z_0p'(z_0)\Phi(z_0).
	\end{equation}
	Since $h\in\mathcal{Q},$ therefore $h'(\zeta_0)\neq 0$ exists. Let $p\nprec h,$ then from~\cite[Lemma~2.2]{geo} and~\cite[Lemma~2.3]{harmonic}, we have $z_0\in\mathbb{D}\backslash\{0\}$ and $\zeta_0\in\partial\mathbb{D}\backslash E(h)$ such that
	\begin{equation}\label{7}
		p(\mathbb{D}_{|z_0|})\subset h(\mathbb{D}),\quad p(z_0)=h(\zeta_0)
	\end{equation}
	and
	\begin{equation}\label{8}
		z_0p'(z_0)=m\zeta_0h'(\zeta_0)\quad\text{for some}\;m\geq 1.
	\end{equation}
	Using~\eqref{7} and~\eqref{8} in~\eqref{6}, we obtain
	\begin{equation*}
		x=h(\zeta_0)\quad\text{and}\quad y=h(\zeta_0)\Theta(z_0)+m\zeta_0h'(\zeta_0)\Phi(z_0).
	\end{equation*}
	Let $\mathbb{P}$ be an open half plane, which supports the convex domain $h(\mathbb{D})$ at $h(\zeta_0).$ So
	\begin{equation}\label{x}
		x=h(\zeta_0)\in\overline{\mathbb{P}}\quad\text{and}\quad h(\mathbb{D})\cap\mathbb{P}=\phi.
	\end{equation}
	We observe that
	\begin{eqnarray}\label{y}
		y = h(\zeta_0)\Theta(z_0)+m\zeta_0h'(\zeta_0)\Phi(z_0) =h(\zeta_0)+\zeta_0h'(\zeta_0)\Psi(z_0),
	\end{eqnarray}
	where
	\begin{equation}\label{9}
		\Psi(z):=m\Phi(z_0)+\dfrac{h(\zeta_0)}{\zeta_0h'(\zeta_0)}(\Theta(z_0)-1).
	\end{equation}
	Clearly~\eqref{hypo} together with the fact that $m\geq 1$ and $\RE\Phi(z)>0$ implies $\RE\Psi(z)>0.$ Using this along with~\eqref{y}, we can say that $y\in\mathbb{P}.$ For $x,y\in\mathbb{P}$ and $\delta\in(0,1],$ it follows from~\cite[Lemma~2.1]{harmonic} that the harmonic mean of $x$ and $y,$ $H_{\delta;\Theta,\Phi,p}(z_0)\in\overline{\mathbb{P}}$ provided $y+\delta(x-y)\neq 0.$ Taking into account~\eqref{x}, it follows that $H_{\delta;\Theta,\Phi,p}(z_0)\notin h(\mathbb{D}),$ which contradicts the hypothesis and thus the result holds in this case.\\
	\indent For the case when $y+\delta(x-y)=0,$ we have $P_{1-\delta;\Theta,\Phi,p}(z_0)=0.$ By Definition~\ref{defH}, we know that the limit
	\begin{equation*}
		H_{\delta;\Theta,\Phi,p}(z_0)= \lim_{\mathbb{D}\backin\zeta\rightarrow z_0}\dfrac{P_{0;\Theta,\Phi,p}(\zeta)P_{1;\Theta,\Phi,p}(\zeta)}{P_{1-\delta;\Theta,\Phi,p}(\zeta)}
	\end{equation*}
	is finite and so $P_{0;\Theta,\Phi,p}(z_0)P_{1;\Theta,\Phi,p}(z_0)=xy=0.$ First let us suppose that $x=0,$ which implies $(1-\delta)y=0.$ Since $x=0,$ $y$ reduces to $y=m\zeta_0h'(\zeta_0)\Phi(z_0).$ Clearly $y$ can not be zero as $h'(\zeta_0)\neq 0$ and $\RE \Phi(z)>0.$ Hence $x=0$ if and only if $\delta=1.$ Now we observe from Definition~\ref{defH},
	\begin{equation*}
		H_{1;\Theta,\Phi,p}(z_0)=P_{1;\Theta,\Phi,p}(z_0)=y\in\mathbb{P},
	\end{equation*}
	which further implies $H_{1;\Theta,\Phi,p}(z_0)\notin h(\mathbb{D}).$ This is a contradiction to hypothesis and the result follows. Next let us suppose that $y=0,$ then we have $\delta x=0.$ Since $\delta\in(0,1],$ it follows that $x=0$ and thus $y$ becomes $y=m\zeta_0h'(\zeta_0)\Phi(z_0),$ which can never be equal to 0. Therefore such a case is never possible. This completes the proof.
\end{proof}
\begin{remark}
	If we take $t=1/2,$ $\Theta(z)=1$ and $\Phi(z)=1$ in $H_{t;\Theta,\Phi,p}(z),$ from Definition~\ref{defH} we can say that it reduces to
	\begin{equation*}
		H_{1/2;1,1,p}(z)=\dfrac{2p(z)(p(z)+zp'(z))}{2p(z)+zp'(z)}=:P(z).
	\end{equation*}
	Further if we take $h(z)=(1+z)/(1-z)$ in Lemma~\ref{mainn}, it reduces to a result of Kanas and Tudor~\cite[Theorem~2.1]{kanasbmms}.
\end{remark}
\begin{remark}
	If $h(z)=((1+z)/(1-z))^{\gamma},$ where $\gamma\in(0,1],$ then for $t=1/2,$ $\Theta(z)=1$ and $\Phi(z)=1,$ Lemma~\ref{mainn} reduces to~~\cite[Theorem~2.6]{kanasbmms}.
\end{remark}
Some applications of the above theorem are discussed in the next section. So far in this direction, authors have proved differential subordination implications of the form
\begin{equation*}
	\psi(p(z),zp'(z))\prec h(z)\Rightarrow p(z)\prec h(z)
\end{equation*}
and the above theorem generalizes many such results in case of harmonic mean. Henceforth, we consider differential subordinations of the form
\begin{equation*}
	\psi(p(z),zp'(z))\prec h(z)\Rightarrow p(z)\prec q(z)
\end{equation*}
for different choices of $h,$ $q$ and $\psi.$ We enlist below a few examples.
\begin{example}
Let $p(z)=1+a_1z+a_2z^2+\cdots$ be analytic in $\mathbb{D}$ with $p(z)\not\equiv1.$ Then $p(z)\prec e^z,$ whenever
\begin{equation*}
  \dfrac{2p(z)(p(z)+zp'(z))}{2p(z)+zp'(z)}\prec \phi_i(z)\;(i=1,2,..,5),
\end{equation*}
where $\phi_1(z)=\sqrt{1+z},\;\phi_2(z)=2/(1+e^{-z}),\;\phi_3(z)=z+\sqrt{1+z^2},\;\phi_4(z)=1+\sin{z}$ and $\phi_5(z)=1+4z/3+2z^2/3.$
\end{example}
\begin{proof}
Let $q(z)=e^z$ and $\Omega_i=\phi_i(\mathbb{D})\;(i=1,2..,5)$. Suppose that  $\psi:\mathbb{C}^2\times\mathbb{D}\rightarrow\mathbb{C}$ be a function defined by $\psi(a,b;z)=2a(a+b)/(2a+b).$ Using the admissibility conditions for $e^z$ given by Naz et al.~\cite{adiba}, it is sufficient to prove that $\psi\in\Psi[\Omega_i,e^z],$ or equivalently, $\psi(r,s;z)\notin\Omega,$ where $r=e^{e^{i\theta}}$ and $s=me^{i\theta}e^{e^{i\theta}}=mre^{i\theta}\;(-\pi\leq\theta\leq\pi).$ We observe that
 \begin{equation*}
  \psi(r,s;z)=\dfrac{2r(r+s)}{2r+s}=\dfrac{2r(r+mre^{i\theta})}{2r+mre^{i\theta}}=2e^{e^{i\theta}}\left(1-\dfrac{1}{2+me^{i\theta}}\right).
\end{equation*}
Therefore
\begin{eqnarray*}
  \RE\psi(r,s;z)&=&\dfrac{2e^{\cos{\theta}} \left(\left(m^2+3 m \cos{\theta}+2\right) \cos{(\sin{\theta})}-m \sin {\theta} \sin{(\sin{\theta})}\right)}{m^2+4 m \cos (\theta)+4}\\
  &=:&g(\theta),
\end{eqnarray*}
which at $\theta=0$ becomes $g(0)=2e(m^2+3m+2)/(5+4m).$ Since $m\geq 1,$ we have $g(0)\geq 4e/3.$ Thus it is easy to conclude that $\psi(r,s,z)\notin\Omega_i\;(i=1,2,..,5)$ and the result follows.
\end{proof}
\begin{example}
Let $p(z)=1+a_1z+a_2z^2+\cdots$ be analytic $\mathbb{D}$ with $p(z)\not\equiv1.$ Then $p(z)\prec \sqrt{1+z},$ whenever
\begin{equation*}
  \dfrac{2p(z)(p(z)+zp'(z))}{2p(z)+zp'(z)}\prec\dfrac{2}{1+e^{-z}}.
\end{equation*}
\end{example}
\begin{proof}
For $\phi(z)=2/(1+e^{-z}),$ let us take $\Omega=\phi(\mathbb{D})$ and $q(z)=\sqrt{1+z}$. Now let  $\psi:\mathbb{C}^2\times\mathbb{D}\rightarrow\mathbb{C}$ be defined as $\psi(a,b;z)=2a(a+b)/(2a+b).$ Then by using the admissibility conditions for $\sqrt{1+z}$ given by Madaan et al.~\cite{vibha}, it suffices to show that $\psi\in\Psi[\Omega,\sqrt{1+z}],$ which is equivalent to, $\psi(r,s;z)\notin\Omega,$ where $r=\sqrt{2\cos{2\theta}}e^{i\theta}$ and $s=me^{3i\theta}/(2\sqrt{2\cos{2\theta}})=me^{2i\theta}/(2r)\;(-\pi/4\leq\theta\leq\pi/4).$ We observe that
 \begin{equation*}
  \psi(r,s;z)=\dfrac{2r(r+s)}{2r+s}=\dfrac{2r\left(r+\dfrac{me^{2i\theta}}{2r}\right)}{2r+\dfrac{me^{2i\theta}}{2r}}=2\sqrt{2\cos{2\theta}}e^{i\theta}\left(\dfrac{m+4\cos{2\theta}}{m+8\cos{2\theta}}\right).
\end{equation*}
Therefore
\begin{equation*}
  \RE\psi(r,s;z)=2\sqrt{2\cos{2\theta}}\cos{\theta}\left(\dfrac{m+4\cos{2\theta}}{m+8\cos{2\theta}}\right),
\end{equation*}
which is an increasing function of $m.$ So for $m\geq 1,$ we have
\begin{equation*}
 \RE\psi(r,s;z)\geq2\sqrt{2\cos{2\theta}}\cos{\theta}\left(\dfrac{1+4\cos{2\theta}}{1+8\cos{2\theta}}\right)=:g(\theta),
\end{equation*}
which at $\theta=0$ becomes $g(0)=10\sqrt{2}/9\approx 1.57.$ Since $\max\RE(2/(1+e^{-z}))\leq 2e/(1+e)\approx 1.46,$ it is easy to conclude that $\psi(r,s,z)\notin\Omega$ and the result follows.
\end{proof}
\begin{example}
Let $p(z)=1+a_1z+a_2z^2+\cdots$ be analytic $\mathbb{D}$ with $p(z)\not\equiv1.$ Then $p(z)\prec 2/(1+e^{-z}),$ whenever
\begin{equation*}
  \dfrac{2p(z)(p(z)+zp'(z))}{2p(z)+zp'(z)}\prec\sqrt{1+z}.
\end{equation*}
\end{example}
\begin{proof}
Let us suppose $q(z)=2/(1+e^{-z})$ and $\Omega=\phi(\mathbb{D}),$ with $\phi(z)=\sqrt{1+z}$. Also, let  $\psi:\mathbb{C}^2\times\mathbb{D}\rightarrow\mathbb{C}$ is a function given by $\psi(a,b;z)=2a(a+b)/(2a+b).$ By applying the admissibility conditions for $2/(1+e^{-z})$ given by Kumar and Goel~\cite{pri3}, it is sufficient to prove that $\psi\in\Psi[\Omega,2/(1+e^{-z})],$ which means $\psi(r,s;z)\notin\Omega,$ with $r=2/(1+e^{-e^{i\theta}})$ and $s=(2me^{i\theta}e^{-e^{i\theta}})/(1+e^{-e^{i\theta}})^2=(mre^{i\theta}e^{-e^{i\theta}})/(1+e^{-e^{i\theta}})\;(-\pi\leq\theta\leq\pi).$ We observe that
 \begin{equation*}
  \psi(r,s;z)=\dfrac{2r(r+s)}{2r+s}=\dfrac{4}{1+e^{-e^{i\theta}}}\left(\dfrac{1+e^{-e^{i\theta}}+me^{i\theta}e^{-e^{i\theta}}}{2+2e^{-e^{i\theta}}+me^{i\theta}e^{-e^{i\theta}}}\right).
\end{equation*}
Then
\begin{equation*}
  \RE\psi(r,s;z)=\dfrac{N(\theta)}{D(\theta)}=:g(\theta),
\end{equation*}
where
\begin{eqnarray*}
N(\theta)&=&m^2 e^{\cos \theta}+m^2 \cos (\sin \theta)+5 m e^{\cos \theta} \cos \theta+3 m e^{2 \cos \theta} \cos (\theta-\sin \theta)\\
& &+m \cos (\theta-\sin \theta)+2 m \cos (\theta+\sin \theta)+m e^{\cos \theta} \cos (\theta-2 \sin \theta)+4 e^{\cos \theta}\\
& &+2 e^{3 \cos \theta}+6 e^{2 \cos \theta} \cos (\sin \theta)+2 \cos (\sin \theta)+2 e^{\cos \theta} \cos (2 \sin \theta)
\end{eqnarray*} and \begin{eqnarray*}
D(\theta)&=&(1+e^{2\cos{\theta}}+2e^{\cos{\theta}}\cos{(\sin{\theta})})(4+m^2+4e^{2\cos{\theta}}+4m\cos{\theta}\\
& &+8e^{\cos{\theta}}\cos{(\sin{\theta})}+4me^{\cos{\theta}}\cos{(\theta-\sin{\theta})}).
\end{eqnarray*}
We observe that $g(0)=4e(1+e+m)/(1+e)(2+2e+m)$ and $m\geq 1,$ so we have $g(0)\geq 4e(2+e)/(1+e)(3+2e)\approx 1.635.$ Thus it is easy to conclude that $\psi(r,s;z)\notin\Omega$ and the result follows.
\end{proof}
 \begin{theorem}
 Let $-1\leq B<A\leq 1,$ $-1\leq E<D\leq 1,$ $L:=k+2,$ $M:=2(A+B)+k(2A-B),$ $N:=2AB,$ $G:=2(E-D),$ $H:=2AE(k+2)-2BE(k-1)-AD(k+2)+BD(k-4),$ $I:=2A^2E(k+2)-2ABE(k-2)-ABD(k+4)+B^2D(k-2)$ and $J:=2A^2BE-2AB^2D$ with
\begin{equation}\label{3}
  2E(1+A)-D(1+B))>0.
\end{equation}
 In addition, assume
\begin{equation}\label{2}
GH+HI-3GJ+IJ+12GJ\geq 4|GI+HJ|\quad (k\geq 1)
\end{equation}
and
\begin{equation}\label{4}
  3+2AB+D(B+1)(A(2B+3)+B+2)\geq 2 E(A+1)(A(B+2)+1)+|4A+B|.
\end{equation}
   Let $p(z)=1+a_1z+a_2z^2+\cdots$ be analytic in $\mathbb{D}$ with $p(z)\not\equiv 1.$ Then
   \begin{equation*}
     \dfrac{2p(z)(p(z)+zp'(z))}{2p(z)+zp'(z)}\prec\dfrac{1+Dz}{1+Ez} \quad\Rightarrow\quad p(z)\prec\dfrac{1+Az}{1+Bz}.
   \end{equation*}
 \end{theorem}
 \begin{proof}
  Define $P(z)$ by
  \begin{equation*}
    P(z):=\dfrac{2p(z)(p(z)+zp'(z))}{2p(z)+zp'(z)}\quad\text{and}\quad\omega(z):=\dfrac{p(z)-1}{A-Bp(z)},
  \end{equation*}
  or equivalently $p(z)=(1+A\omega(z))/(1+B\omega(z)).$ Then $\omega(z)$ is clearly meromorphic in $\mathbb{D}$ and $\omega(0)=0.$ We need to show that $|\omega(z)|<1$ in $\mathbb{D}.$ We have
 \begin{equation*}
   P(z)=\dfrac{2\left(\dfrac{1+A\omega(z)}{1+B\omega(z)}\right)\left(\dfrac{1+A\omega(z)}{1+B\omega(z)}+\dfrac{(A-B)z\omega'(z)}{(1+B\omega(z))^2}\right)}{2\left(\dfrac{1+A\omega(z)}{1+B\omega(z)}\right)+\dfrac{(A-B)z\omega'(z)}{(1+B\omega(z))^2}}.
 \end{equation*}
 Therefore
{\small \begin{equation*}
   \dfrac{P(z)-1}{D-E P(z)}=-\dfrac{(A-B)(2\omega(z)(1+A\omega(z))(1+B\omega(z))+(1+2A\omega(z)-B\omega(z))z\omega'(z))}{2(1+A\omega(z))(1+B\omega(z))\Phi_1(z)+(A-B)\Phi_2(z)z\omega'(z)},
 \end{equation*}}
 where
 \begin{equation*}
 \Phi_1(z)=E(1+A\omega(z))-D(1+B\omega(z))\;\text{and}\;\Phi_2(z)=2E(1+A\omega(z))-D(1+B\omega(z)).
 \end{equation*}
On the contrary if there exists a point $z_0\in\mathbb{D}$ such that $\max_{|z|\leq|z_0|}|\omega(z)|=|\omega(z_0)|=1,$ then by~\cite[Lemma~1.3]{rush}, there exists $k\geq1$ such that $z_0\omega'(z_0)=k\omega(z_0).$ Let $\omega(z_0)=e^{i\theta},$ then we have
 \begin{eqnarray*}
   \left|\dfrac{P(z_0)-1}{D-EP(z_0)}\right|=(A-B)\left|\dfrac{L+Me^{i\theta}+Ne^{2i\theta}}{G+He^{i\theta}+Ie^{2i\theta}+Je^{3i\theta}}\right|.
 \end{eqnarray*}
 We observe that
 \begin{equation*}
   |L+Me^{i\theta}+Ne^{2i\theta}|^2=L^2+M^2+N^2-2LN+2(L+N)M\cos{\theta}+4LN\cos^2{\theta}.
 \end{equation*}
Choose $t:=\cos{\theta}\in[-1,1].$ Since
 \begin{equation*}
   \min\{at^2+bt+c:-1\leq t\leq1\}=\begin{cases}
                                      \dfrac{4ac-b^2}{4a}, & \quad\text{if}\;a>0\; \text{and}\;|b|<2a \\
                                      a-|b|+c, & \quad\text{otherwise,}
                                    \end{cases}
 \end{equation*}
we have $|L+Me^{i\theta}+Ne^{2i\theta}|^2\geq (L-|M|+N)^2.$ Next we consider
\begin{eqnarray*}
  |G+He^{i\theta}+Ie^{2i\theta}+Je^{3i\theta}|^2&=&G^2+H^2+I^2+J^2-2GI-2HJ+(2GH\\
& &+2HI-6GJ+2IJ)\cos{\theta}+(4GI+4HJ)\cos^2{\theta}\\
& &+8GJ\cos^3{\theta},
\end{eqnarray*}
which is an increasing function of $t=\cos{\theta}\in[-1,1]$ in view of~\eqref{2}. Thus we have $|G+He^{i\theta}+Ie^{2i\theta}+Je^{3i\theta}|^2\leq (G+H+I+J)^2.$ Therefore
\begin{equation*}
  \left|\dfrac{P(z_0)-1}{D-EP(z_0)}\right|^2\geq \left(\dfrac{L-|M|+N}{G+H+I+J}\right)^2=:\psi(k) ,\end{equation*}
which in view of~\eqref{3} is an increasing function of $k.$ So we have $\psi(k)\geq\psi(1)$ and therefore
\begin{equation*}
\left|\dfrac{P(z_0)-1}{D-EP(z_0)}\right|\geq \dfrac{3+2AB-|4A+B|}{2E(A+1)(A(B+2)+1)-D(B+1)(A(2 B+3)+B+2)},
\end{equation*}
which in view of~\eqref{4} is greater than or equal to $1.$ This contradicts that $P(z)\prec(1+Dz)(1+Ez)$ and that completes the proof.
\end{proof}
\begin{note}
The fact that the equations~\eqref{3},~\eqref{2} and~\eqref{4} hold simultaneously is validated by the following set of values: $A=3/8,\;B=0,\;D=1,\;E=123/128$.
\end{note}
In the next two results, we consider a combination of harmonic mean, geometric mean and arithmetic mean of $p(z)$ and $p(z)+zp'(z)/p(z)$. Note that these results involve some constants given by $x=\IM p(z_0)$ and $y=z_0p'(z_0)/m,$ where $z_0$ and $m$ are as defined in Lemma~\ref{jack} for $f(z)=(p(z)-1)/(1-2\alpha+p(z)).$
\begin{theorem}\label{htheo1}
	Let $\gamma\in[0,1],$ $\alpha\in[0,1),$ $\delta\in[1,2],$ $\beta\in[0,1)$ and $\rho\in[0,1]$ be such that $\rho\geq \alpha(1+2\alpha),$ whenever $\alpha\in[0,1/2].$ Also, let $p$ be an analytic function with $p(0)=1$ and
\begin{equation*}
\RE\left(\gamma(p(z))^{\delta}+(1-\gamma)\dfrac{p(z)+\dfrac{zp'(z)}{p(z)}}{1+\rho\dfrac{zp'(z)}{p^2(z)}}\right)>\beta,
\end{equation*}
for $\beta\geq \gamma\alpha+(1-\gamma)\beta_0,$ where $\beta_0$ is given as follows:\\
\begin{equation}\label{h4}
	\beta_0=\begin{cases}\alpha\dfrac{(1+\alpha)(1-2\alpha)}{\rho(1-\alpha)-2\alpha^2},\quad &\text{if}\;I_1\;\text{holds}\\
		\alpha\quad &\text{if}\;(\sim I_1)\wedge I_2\;\text{holds}\\
		\alpha+\dfrac{\rho(1-\rho)(1-\alpha)(2(1-\alpha)+\rho)}{16\alpha(2\alpha^2-\rho(1-\alpha))} & \text{if}\;(\sim I_1)\wedge(\sim I_2)\wedge I_3\; \text{holds}\\
		\alpha+\dfrac{\rho(1-\rho)}{16\alpha(1-\alpha)}\left(\dfrac{2\alpha^2-\rho(1-\alpha)}{2(1-\alpha)+\rho}\right)&\text{if}\;(\sim I_1)\wedge(\sim I_2)\wedge(\sim I_3)\;\text{holds},
	\end{cases}
\end{equation}
with
\begin{equation}\label{h14}
I_1: 0\leq\alpha\leq 1/2,\;\;I_2:\alpha^2-x^2+\rho my>0\;\;\text{and}\;\;I_3:x^2\geq \tfrac{(2\alpha^2-\rho(1-\alpha))(1-\alpha)}{2(1-\alpha)+\rho},
\end{equation}
provided $x>0$ and $my\leq-\tfrac{(1-\alpha)^2+x^2}{2(1-\alpha)}.$ Then $\RE p(z)>\alpha.$
\end{theorem}
\begin{proof}
If we let $q(z)=(1+(1-2\alpha)z)/(1-z),$ then it suffices to show that $p\prec q.$ Let us suppose that $p$ is not subordinate to $q.$ Then by~\cite[Lemma~2.2d,p~24]{ds} and~\cite[Lemma~2.2f,p~26]{ds}, there exist $z_0\in\mathbb{D},$ $\zeta_0\in\partial\mathbb{D}\backslash\{1\}$ and $m\geq 1$ such that
\begin{equation}\label{h1}
	p(z_0)=q(\zeta_0)=\alpha+ix\quad\text{and}\quad z_0p'(z_0)=m\zeta_0 q'(\zeta_0)=:my\leq -\dfrac{(1-\alpha)^2+x^2}{2(1-\alpha)}.
\end{equation}
Consequently, we have
\begin{equation}\label{h6}
\gamma(p(z_0))^{\delta}+(1-\gamma)\dfrac{p(z_0)+\dfrac{z_0p'(z_0)}{p(z_0)}}{1+\rho\dfrac{z_0p'(z_0)}{p^2(z_0)}}=\gamma(q(\zeta_0))^{\delta}+(1-\gamma)\dfrac{q(\zeta_0)+\dfrac{m\zeta_0q'(\zeta_0)}{q(\zeta_0)}}{1+\rho\dfrac{m\zeta_0q'(\zeta_0)}{q^2(\zeta_0)}}.
\end{equation}
Now let $E(\delta)=(q(\zeta_0))^{\delta}$ and $L=\{E(\delta):\delta\in[1,2]\}.$ Geometrically, L represents an arc of the logarithmic spiral joining the points $E(1)=q(\zeta_0)$ and $E(2)=(q(\zeta_0))^2.$ We know that L cuts each radial halfline at an angle, which is constant. Clearly, $\arg(q(\zeta_0))^{\delta}$ is an increasing function of $\delta$ and thus, L is in the closed halfplane containing the origin and bounded by the line $\RE z=\RE E(1).$ As a result, we have
\begin{equation}\label{h7}
	\RE (q(\zeta_0))^{\delta}\leq \alpha,\quad\delta\in[1,2].
\end{equation}
We observe that
\begin{equation}\label{h5}
\dfrac{q(\zeta_0)+\dfrac{m\zeta_0q'(\zeta_0)}{q(\zeta_0)}}{1+\rho\dfrac{m\zeta_0q'(\zeta_0)}{q^2(\zeta_0)}}=q(\zeta_0)+(1-\rho)\dfrac{m\zeta_0q'(\zeta_0)}{q(\zeta_0)+\rho m\zeta_0 q'(\zeta_0)/q(\zeta_0)}.
\end{equation}
Using~\eqref{h1} and~\eqref{h5}, we have
\begin{eqnarray}\nonumber\label{hh2}
	\RE\left(\dfrac{q(\zeta_0)+\dfrac{m\zeta_0q'(\zeta_0)}{q(\zeta_0)}}{1+\rho\dfrac{m\zeta_0q'(\zeta_0)}{q^2(\zeta_0)}}\right)&=&\RE \left(\alpha+ix+(1-\rho)\dfrac{my}{\alpha+ix+\tfrac{\rho my}{\alpha+ix}}\right)\\ \nonumber
	& =& \alpha+(1-\rho)\dfrac{my\alpha(\alpha^2+x^2+\rho m y)}{(\alpha^2-x^2+\rho m y)^2+4\alpha^2x^2}.\\
	\end{eqnarray}
Since $my<0,$ we may observe that
\begin{equation}\label{h2}
\RE\left(\dfrac{q(\zeta_0)+\dfrac{m\zeta_0q'(\zeta_0)}{q(\zeta_0)}}{1+\rho\dfrac{m\zeta_0q'(\zeta_0)}{q^2(\zeta_0)}}\right)\leq  \alpha+(1-\rho)\dfrac{my\alpha(\alpha^2-x^2+\rho m y)}{(\alpha^2-x^2+\rho m y)^2+4\alpha^2x^2}.
\end{equation}
Case~(i): $0\leq \alpha\leq 1/2.$\\
From hypothesis, we have $\rho\geq\alpha(1+2\alpha)\geq 2\alpha^2/(1-\alpha)$ and from~\eqref{h1}, we have $my\leq -\tfrac{(1-\alpha)^2+x^2}{2(1-\alpha)}.$ Therefore
\begin{equation*}
	\alpha^2-x^2+\rho my\leq \dfrac{2\alpha^2-\rho(1-\alpha)}{2}-x^2\dfrac{2(1-\alpha)+\rho}{2(1-\alpha)}\leq 0.
\end{equation*}
Now using the fact that $\alpha^2x^2>0$ and that $my\leq -\tfrac{(1-\alpha)^2+x^2}{2(1-\alpha)}\leq -\tfrac{1-\alpha}{2},$ we obtain from~\eqref{h2},
\begin{eqnarray*}
	\RE\left(\dfrac{q(\zeta_0)+\dfrac{m\zeta_0q'(\zeta_0)}{q(\zeta_0)}}{1+\rho\dfrac{m\zeta_0q'(\zeta_0)}{q^2(\zeta_0)}}\right)&\leq&  \alpha+(1-\rho)\dfrac{my\alpha(\alpha^2-x^2+\rho m y)}{(\alpha^2-x^2+\rho m y)^2}\\
	&=&\alpha+(1-\rho)\dfrac{\alpha}{\tfrac{\alpha^2-x^2}{my}+\rho}\\
	&\leq& \alpha +(1-\rho)\dfrac{\alpha}{\rho-\tfrac{2\alpha^2}{1-\alpha}}\\
	&=&\alpha\dfrac{(1+\alpha)(1-2\alpha)}{\rho (1-\alpha)-2\alpha^2}.
\end{eqnarray*}
Case~(ii): $1/2<\alpha<1$ and $\alpha^2-x^2+\rho my>0.$\\
Since $\alpha>1/2,$ we have $2\alpha^2/(1-\alpha)>1$ and therefore, $\rho<2\alpha^2/(1-\alpha).$ Also, since $my<0$ and $\alpha^2-x^2+\rho my>0,$ we have
\begin{eqnarray*}
		\RE\left(\dfrac{q(\zeta_0)+\dfrac{m\zeta_0q'(\zeta_0)}{q(\zeta_0)}}{1+\rho\dfrac{m\zeta_0q'(\zeta_0)}{q^2(\zeta_0)}}\right)&\leq&\alpha+(1-\rho)\dfrac{my\alpha(\alpha^2-x^2+\rho m y)}{(\alpha^2-x^2+\rho m y)^2+4\alpha^2x^2}\leq \alpha.
\end{eqnarray*}
Case~(iii): $1/2<\alpha<1,$ $\alpha^2-x^2+\rho my\leq0$ and $ x^2\geq \tfrac{(2\alpha^2-\rho(1-\alpha))(1-\alpha)}{2(1-\alpha)+\rho}.$
From~\eqref{hh2}, we have
\begin{eqnarray}\nonumber\label{h3}
	\RE\left(\dfrac{q(\zeta_0)+\dfrac{m\zeta_0q'(\zeta_0)}{q(\zeta_0)}}{1+\rho\dfrac{m\zeta_0q'(\zeta_0)}{q^2(\zeta_0)}}\right)&\leq&\alpha+(1-\rho)my\dfrac{\alpha(\alpha^2+x^2+\rho my)}{(\alpha^2-x^2+\rho my)^2+4\alpha^2 x^2}\\ \nonumber
	&\leq& \alpha+(1-\rho)\dfrac{\alpha\rho m^2y^2}{(\alpha^2-x^2+\rho my)^2+4\alpha^2 x^2}\\ \nonumber
	&\leq&\alpha+(1-\rho)\dfrac{\alpha\rho m^2y^2}{4\alpha^2 x^2}\\
	&=&\alpha+(1-\rho)\dfrac{\rho m^2y^2}{4\alpha x^2}.
\end{eqnarray}
Using the inequality $my\leq-\tfrac{(1-\alpha)^2+x^2}{2(1-\alpha)}\leq -\tfrac{1-\alpha}{2}$ and the condition on $x^2$ in~\eqref{h3}, we obtain
\begin{eqnarray*}
	\RE\left(\dfrac{q(\zeta_0)+\dfrac{m\zeta_0q'(\zeta_0)}{q(\zeta_0)}}{1+\rho\dfrac{m\zeta_0q'(\zeta_0)}{q^2(\zeta_0)}}\right)&\leq& \alpha+(1-\rho)\dfrac{\rho}{4\alpha}\left(\dfrac{1-\alpha}{2}\right)^2\left(\dfrac{2(1-\alpha)+\rho}{(2\alpha^2-\rho(1-\alpha))(1-\alpha)}\right)\\
&=&\alpha+\dfrac{\rho(1-\rho)(1-\alpha)(2(1-\alpha)+\rho)}{16\alpha(2\alpha^2-\rho(1-\alpha))}.
\end{eqnarray*}
Case~(iv): $1/2<\alpha<1,$ $\alpha^2-x^2+\rho my\leq0$ and $ x^2< \tfrac{(2\alpha^2-\rho(1-\alpha))(1-\alpha)}{2(1-\alpha)+\rho}.$\\
Considering $my\leq-\tfrac{(1-\alpha)^2+x^2}{2(1-\alpha)}\leq -\tfrac{x^2}{2(1-\alpha)}$ and using it in~\eqref{h3} along with the condition on $x^2$, we get
\begin{eqnarray*}
		\RE\left(\dfrac{q(\zeta_0)+\dfrac{m\zeta_0q'(\zeta_0)}{q(\zeta_0)}}{1+\rho\dfrac{m\zeta_0q'(\zeta_0)}{q^2(\zeta_0)}}\right)&\leq& \alpha+(1-\rho)\dfrac{\rho}{4\alpha x^2}\left(\dfrac{-x^2}{2(1-\alpha)}\right)^2\\
	&=&\alpha+\dfrac{\rho(1-\rho)x^2}{16\alpha(1-\alpha)^2}\\
	&\leq& \alpha+\dfrac{\rho(1-\rho)}{16\alpha(1-\alpha)}\left(\dfrac{2\alpha^2-\rho(1-\alpha)}{2(1-\alpha)+\rho}\right).
\end{eqnarray*}
Combining all the cases we obtain
\begin{equation}\label{ha7}
	\RE\left(\dfrac{q(\zeta_0)+\dfrac{m\zeta_0q'(\zeta_0)}{q(\zeta_0)}}{1+\rho\dfrac{m\zeta_0q'(\zeta_0)}{q^2(\zeta_0)}}\right)\leq\beta_0,	
\end{equation}
where $\beta_0$ is given by~\eqref{h4}. From~\eqref{h6},~\eqref{h7} and~\eqref{ha7}, we have
\begin{equation*}
	\RE\left(\gamma(p(z_0))^{\delta}+(1-\gamma)\dfrac{p(z_0)+\dfrac{z_0p'(z_0)}{p(z_0)}}{1+\rho\dfrac{z_0p'(z_0)}{p^2(z_0)}}\right)\leq \gamma\alpha+(1-\gamma)\beta_0,
\end{equation*}
which contradicts the hypothesis. Hence the result follows.
\end{proof}
\begin{remark}
	If we take $\gamma=0$ and $\alpha\in[0,1/2],$ the above result reduces to a result of Kanas~\cite[Theorem~2.4]{kanas2017}. Infact, we extended this result for $\alpha\in[0,1).$
\end{remark}
\begin{remark}
	By taking $\gamma=0,\;\rho=1/2$ and $\alpha=0,$ we obtain a result of Kanas and Tudor~\cite[Theorem~2.3]{kanasbmms}
\end{remark}
\begin{theorem}\label{genels}
	Let $\gamma\in[0,1],$ $\alpha\in[0,1),$ $\mu\in[0,1],$ $\delta\in[1,2],$ $\beta\in[0,1)$ and $\rho\in[0,1]$ be such that $\rho\geq \alpha(1+2\alpha),$ whenever $\alpha\in[0,1/2].$ Also, let $p$ be an analytic function with $p(0)=1$ and
\begin{equation*}
	\RE\left(\gamma(p(z))^{\delta}+(1-\gamma)\dfrac{(p(z))^{\mu}\left(p(z)+\dfrac{zp'(z)}{p(z)}\right)^{1-\mu}}{1+\rho\dfrac{zp'(z)}{p^2(z)}}\right)>\beta,
\end{equation*}
for $\beta\geq \gamma\alpha+(1-\gamma)\beta_1,$ where $\beta_1$ is defined as follows:\\
{\small \begin{equation*}
		\beta_1=\begin{cases}\alpha&I_4\;\text{holds}\\
			\alpha+\dfrac{2\alpha\rho^2(1-2\alpha)}{(\rho-2(1-\alpha))^2(\rho(1-\alpha)-2\alpha^2)},\quad &I_1\wedge (\sim I_4)\;\text{holds}\\
		\alpha+\dfrac{\alpha\rho(1-\alpha)(4\alpha^2-\rho(1-\alpha))}{4(2\alpha^2-\rho(1-\alpha))^2}\quad &(\sim I_1)\wedge I_2\wedge(\sim I_4)\;\text{holds}\\
		\alpha+\dfrac{\alpha\rho(1-\alpha)(2(1-\alpha)+\rho)}{2(4\alpha^2(1-\alpha)+\rho(2\alpha-1))} & (\sim I_1)\wedge(\sim I_2)\wedge I_3\wedge({\sim I_4})\; \text{holds}\\
		\alpha+\dfrac{\alpha\rho(2\alpha^2-\rho(1-\alpha))}{2(4\alpha^2(1-\alpha)+\rho(2\alpha-1))}&(\sim I_1)\wedge(\sim I_2)\wedge(\sim I_3)\wedge(\sim I_4)\;\text{holds},
\end{cases}
\end{equation*}}
with $I_1,\;I_2$ and $I_3$ given by~\eqref{h14} and $I_4:\alpha^2+x^2+\rho my\leq 0,$ provided $x>0$ and $my\leq-\tfrac{(1-\alpha)^2+x^2}{2(1-\alpha)}.$ Then $\RE p(z)>\alpha.$
\end{theorem}
\begin{proof}
We proceed as done in Theorem~\ref{htheo1} to show that $p\prec q,$ where $q(z)=(1+(1-2\alpha)z)/(1-z).$ For if $p\nprec q,$ then there exists $z_0\in\mathbb{D},$ $\zeta_0\in\partial\mathbb{D}$ and $m\geq1$ such that
\begin{align*}
	\gamma(p(z_0))^{\delta}&+(1-\gamma)\dfrac{(p(z_0))^{\mu}\left(p(z_0)+\dfrac{z_0p'(z_0)}{p(z_0)}\right)^{1-\mu}}{1+\rho\dfrac{z_0p'(z_0)}{p^2(z_0)}}\\
	&=\gamma(q(\zeta_0))^{\delta}+(1-\gamma)\dfrac{(q(\zeta_0))^{\mu}\left(q(\zeta_0)+\dfrac{m\zeta_0q'(\zeta_0)}{q(\zeta_0)}\right)^{1-\mu}}{1+\rho\dfrac{m\zeta_0q'(\zeta_0)}{q^2(\zeta_0)}}.
\end{align*}
From the proof of Theorem~\ref{htheo1}, we have
\begin{equation}\label{h99}
\RE (q(\zeta_0))^{\delta}\leq \alpha.
\end{equation}
Now we set
\begin{equation*}
	E(\mu)=\dfrac{(q(\zeta_0))^{\mu}\left(q(\zeta_0)+\dfrac{m\zeta_0q'(\zeta_0)}{q(\zeta_0)}\right)^{1-\mu}}{1+\rho\dfrac{m\zeta_0q'(\zeta_0)}{q^2(\zeta_0)}}=\dfrac{q(\zeta_0)\left(1+\dfrac{m\zeta_0q'(\zeta_0)}{q^2(\zeta_0)}\right)^{1-\mu}}{1+\rho\dfrac{m\zeta_0q'(\zeta_0)}{q^2(\zeta_0)}}
\end{equation*}
and let $L=\{E(\mu):\mu\in[0,1]\}.$ Note that
\begin{equation*}
	E(0)=\dfrac{q(\zeta_0)+\dfrac{m\zeta_0q'(\zeta_0)}{q(\zeta_0)}}{1+\rho\dfrac{m\zeta_0q'(\zeta_0)}{q^2(\zeta_0)}}\quad\text{and}\quad E(1)=\dfrac{q(\zeta_0)}{1+\rho\dfrac{m\zeta_0q'(\zeta_0)}{q^2(\zeta_0)}}.
\end{equation*}
It was shown in the proof of Theorem~\ref{htheo1} that $\RE E(0)\leq \beta_0,$ where $\beta_0$ is given by~\eqref{h4}. Now we consider $\RE E(1),$ which by using the conditions~\eqref{h1} becomes
\begin{eqnarray*}
\RE E(1)&=&\RE\left(\dfrac{\alpha+ix}{1+\dfrac{\rho my}{(\alpha+ix)^2}}\right)\\
&=&\dfrac{\alpha(\alpha^4+x^4+2\alpha^2x^2+\alpha^2\rho my-3x^2\rho my)}{(\alpha^2-x^2+\rho my)^2+4\alpha^2x^2}\\
&=&\alpha-\dfrac{\alpha\rho my(\alpha^2+x^2+\rho my)}{(\alpha^2-x^2+\rho my)^2+4\alpha^2x^2}.\\
\end{eqnarray*}
It is easy to note that $\RE E(1)\leq\alpha,$ whenever $\alpha^2+x^2+\rho my\leq 0.$ So now let us assume that $\alpha^2+x^2+\rho my>0$ for the rest of the cases. Before we start our first case, we observe that
\begin{eqnarray}\label{h9}\nonumber
-\alpha\rho my\left(\dfrac{\alpha^2+x^2+\rho my}{(\alpha^2-x^2+\rho my)^2+4\alpha^2x^2}\right)&\leq&-\alpha\rho my\left(\dfrac{\alpha^2+x^2}{(\alpha^2-x^2+\rho my)^2+4\alpha^2x^2}\right)\\ \nonumber
&\leq&\dfrac{-\alpha\rho my(\alpha^2+x^2)}{(\alpha^2+x^2+\rho my)^2}\\
&=&\dfrac{-\alpha\rho (\alpha^2+x^2)}{my\left(\tfrac{\alpha^2+x^2}{my}+\rho\right)^2},
\end{eqnarray}
 since $my<0$ and $\alpha^2+x^2+\rho my>0.$\\
Case~(i): $0\leq\alpha\leq 1/2.$\\
From~\eqref{h1}, we have $my\leq -\tfrac{(1-\alpha)^2+x^2}{2(1-\alpha)},$ which further implies
\begin{equation}\label{h8}
\alpha^2+x^2+\rho my\leq \alpha^2-\dfrac{\rho(1-\alpha)}{2}+x^2\left(1-\dfrac{\rho}{2(1-\alpha)}\right).
\end{equation}
The inequality $\alpha\leq1/2$ implies $1-\rho/2(1-\alpha)>0.$ Also from hypothesis, we have $\rho\geq\alpha(1+2\alpha)\geq 2\alpha^2/(1-\alpha),$ which is sufficient to conclude that
\begin{equation*}
	\alpha^2+x^2+\rho my\leq x^2\left(1-\dfrac{\rho}{2(1-\alpha)}\right).
\end{equation*}
Thus
\begin{equation}\label{h10}
	\dfrac{\alpha^2+x^2}{my}+\rho\geq \rho-2(1-\alpha).
\end{equation}
We know that $\alpha^2+x^2+\rho my>0,$ so from~\eqref{h8} we obtain
\begin{equation}\label{h11}
	x^2\geq \dfrac{(\rho(1-\alpha)-2\alpha^2)(1-\alpha)}{2(1-\alpha)-\rho}.
\end{equation}
Using~\eqref{h10} and~\eqref{h11} in~\eqref{h9}, we get
\begin{eqnarray*}
-\alpha\rho my\left(\dfrac{\alpha^2+x^2+\rho my}{(\alpha^2-x^2+\rho my)^2+4\alpha^2x^2}\right)&\leq& \dfrac{2\alpha\rho(1-\alpha)(\alpha^2+x^2)}{x^2(\rho-2(1-\alpha))^2}	\\
&=&\dfrac{2\alpha\rho(1-\alpha)\left(\tfrac{\alpha^2}{x^2}+1\right)}{(\rho-2(1-\alpha))^2}\\
&\leq&\dfrac{2\alpha(1-2\alpha)\rho^2}{(\rho-2(1-\alpha))^2(\rho(1-\alpha)-2\alpha^2)}.
\end{eqnarray*}
Case~(ii): $1/2<\alpha<1$ and $\alpha^2-x^2+\rho my>0$\\
Using the above condition, we obtain
\begin{eqnarray*}
\alpha^2+x^2+\rho my&<&2(\alpha^2+\rho my)\leq 2\left(\alpha^2-\dfrac{\rho(1-\alpha)}{2}\right).
\end{eqnarray*}
Thus we have
\begin{equation}\label{h12}
	\dfrac{\alpha^2+x^2}{my}+\rho\geq \dfrac{-2\left(2\alpha^2-\rho(1-\alpha)\right)}{1-\alpha}.
\end{equation}
Next we observe that
\begin{eqnarray}\label{ha13}
\alpha^2+x^2&\leq& 2\alpha^2+\rho my\leq 2\alpha^2-\dfrac{\rho(1-\alpha)}{2}.
\end{eqnarray}
Using~\eqref{h12} and~\eqref{ha13} in~\eqref{h9}, we get
\begin{eqnarray*}
-\alpha\rho my\left(\dfrac{\alpha^2+x^2+\rho my}{(\alpha^2-x^2+\rho my)^2+4\alpha^2x^2}\right)&\leq& \dfrac{2\alpha\rho\left(2\alpha^2-\dfrac{\rho(1-\alpha)}{2}\right)}{4(1-\alpha)\left(\dfrac{2\alpha^2-\rho(1-\alpha)}{1-\alpha}\right)^2}\\
&=&\dfrac{\alpha\rho(1-\alpha)(4\alpha^2-\rho(1-\alpha))}{4(2\alpha^2-\rho(1-\alpha))^2}.
\end{eqnarray*}
Case~(iii): $1/2<\alpha<1,$ $\alpha^2-x^2+\rho my\leq0$ and $ x^2\geq \tfrac{(2\alpha^2-\rho(1-\alpha))(1-\alpha)}{2(1-\alpha)+\rho}.$\\
Since $my<0,$ $\alpha^2+x^2+\rho my\leq \alpha^2+x^2,$ which further implies
\begin{equation}\label{h13}
	\dfrac{\alpha^2+x^2}{my}+\rho\geq \dfrac{-2(\alpha^2+x^2)}{1-\alpha}.
\end{equation}
Now by using~\eqref{h13} and the condition $my\leq -(1-\alpha)/2$ in~\eqref{h9}, we obtain
\begin{equation*}
-\alpha\rho my\left(\dfrac{\alpha^2+x^2+\rho my}{(\alpha^2-x^2+\rho my)^2+4\alpha^2x^2}\right)\leq \dfrac{\alpha\rho(1-\alpha)}{2(\alpha^2+x^2)},
\end{equation*}
which after applying the condition on $x^2$ becomes
\begin{equation*}
-\alpha\rho my\left(\dfrac{\alpha^2+x^2+\rho my}{(\alpha^2-x^2+\rho my)^2+4\alpha^2x^2}\right)\leq\dfrac{\alpha\rho(1-\alpha)(2(1-\alpha)+\rho)}{2(4\alpha^2(1-\alpha)+\rho(2\alpha-1))}.
\end{equation*}
Case~(iv): $1/2<\alpha<1,$ $\alpha^2-x^2+\rho my\leq0$ and $ x^2< \tfrac{(2\alpha^2-\rho(1-\alpha))(1-\alpha)}{2(1-\alpha)+\rho}.$\\
Proceeding similarly as done in case~(iii) and replacing the condition $my\leq -(1-\alpha)/2$ by $my\leq -x^2/2(1-\alpha),$ we get
\begin{equation*}
-\alpha\rho my\left(\dfrac{\alpha^2+x^2+\rho my}{(\alpha^2-x^2+\rho my)^2+4\alpha^2x^2}\right)\leq \dfrac{\alpha\rho}{2\left(\dfrac{\alpha^2}{x^2}+1\right)(1-\alpha)},	
\end{equation*}
which by using the condition on $x^2$ becomes
\begin{eqnarray*}
	-\alpha\rho my\left(\dfrac{\alpha^2+x^2+\rho my}{(\alpha^2-x^2+\rho my)^2+4\alpha^2x^2}\right)&\leq& \dfrac{\alpha\rho}{2\left(\dfrac{\alpha^2}{x^2}+1\right)(1-\alpha)}\\
	&\leq & \dfrac{\alpha\rho(2\alpha^2-\rho(1-\alpha))}{2(4\alpha^2(1-\alpha)+\rho(2\alpha-1))}.	
\end{eqnarray*}
Combining all the cases, we obtain that $\RE E(1)\leq \beta_1.$ We observe that $L$ represents an arc of logarithmic spiral with end points $E(0)$ and $E(1),$ such that it cuts every radial halfline at a constant angle. Also
\begin{equation*}
\arg E(\mu)=\arg(q(\zeta_0))+(1-\mu)\left(1+\dfrac{m\zeta_0 q'(\zeta_0)}{q^2(\zeta_0)}\right)-\arg\left(1+\rho\dfrac{m\zeta_0 q'(\zeta_0)}{q^2(\zeta_0)}\right)
\end{equation*}
is a decreasing function of $\mu$ and thus, $\arg E(1)\leq \arg E(\mu)\leq \arg E(0),\; \mu\in[0,1].$ So we may conclude that $L$ lies in the closed halfplane containing the origin and determined by the line $\RE z=\RE E(1),$ which means
\begin{equation}\label{h98}
  \RE E(\mu)\leq \beta_1.
\end{equation}
From~\eqref{h99} and~\eqref{h98}, we obtain
\begin{equation*}
	\RE\left(\gamma(p(z))^{\delta}+(1-\gamma)\dfrac{(p(z))^{\mu}\left(p(z)+\dfrac{zp'(z)}{p(z)}\right)^{1-\mu}}{1+\rho\dfrac{zp'(z)}{p^2(z)}}\right)\leq\gamma\alpha+(1-\gamma)\beta_1\leq\beta,
\end{equation*}
which contradicts the hypothesis and hence the result follows.
\end{proof}
\begin{remark}
	Taking $\rho=0,$ we obtain the result \cite[Theorem~2.3]{kanasels}
\end{remark}
\begin{remark}
	By taking $\mu=0,\;\gamma=0,\;\rho=1/2$ and $\alpha=0,$ we obtain the result~\cite[Theorem~2.3]{kanasbmms}.
\end{remark}
\begin{remark}
	For $\gamma=0,\;\rho=0$ and $\alpha=0,$ we obtain a result of Lewandowski et al.~\cite{lewa}.
\end{remark}
\begin{remark}
	For $\delta=1,\;\rho=0,\;\mu=0$ and $\alpha=0,$ we obtain a result of Sakaguchi~\cite{sakaguchi}.
\end{remark}
\section{Applications}
We begin with the applications of Lemma~\ref{mainn}:
\begin{theorem}\label{kcorr}
	Let $\delta\in[0,1],$ $h\in\mathcal{Q}$ with $0\in\overline{h(\mathbb{D})}$ and $\Theta,\Phi\in\mathcal{H}$ are such that $\Theta(0)=1$ and
	\begin{equation}\label{hypo2}
		\RE \Phi(z)\geq 5|\Theta(z)-1|-\RE(\Theta(z)-1),\quad z\in\mathbb{D}.
	\end{equation}
	If $p\in\mathcal{H}(\delta;\Theta,\Phi),$ $p(0)=h(0)$ and $H_{\delta;\Theta,\Phi,p}\prec h,$ then $p\prec h.$
\end{theorem}
\begin{proof}
	In view of Lemma~\ref{mainn}, it is sufficient to show that $\Theta,\;\Phi$ and $h$ satisfies~\eqref{hypo}. Since $h$ is convex and $h(0)=1,$ we can say that $h_1:=h-1\in\mathcal{C}.$ Using Marx Strohh\"{a}cker theorem~\cite{ds}, we have $\RE(\zeta h_1'(\zeta)/h_1(\zeta))>1/2,$ which is equivalent to
	\begin{equation*}
		\left|\dfrac{h_1(\zeta)}{\zeta h_1'(\zeta)}-1\right|\leq 1,
	\end{equation*}
	which means
	\begin{equation}\label{10}
		\left|\dfrac{h(\zeta)}{\zeta h'(\zeta)}-1\right|\leq 1+\dfrac{1}{|h'(\zeta)|}.
	\end{equation}
	Since $h_1\in\mathcal{C},$ we have $|h_1'(z)|\geq 1/(1+r)^2$ on $|z|=r$~\cite[Theorem.9, pp~118]{goodman}. We know that $\zeta\in\partial\mathbb{D},$ so we have $|h'(\zeta)|=|h_1'(\zeta)|\geq 1/4.$ Thus~\eqref{10} reduces to
	\begin{equation}\label{11}
		\left|\dfrac{h(\zeta)}{\zeta h'(\zeta)}-1\right|\leq 5.
	\end{equation}
	Note that if $X,\;Y\in\mathbb{C}$ and $|X-1|\leq K,$ then
	\begin{equation*}
		\RE(X.Y)=\RE Y+\RE Y(X-1))\geq \RE Y-|Y|K.
	\end{equation*}
	Applying this inequality to~\eqref{hypo} and using~\eqref{11}, we get
	\begin{equation*}
		\RE\left( \Phi(z)+\dfrac{h(\zeta)}{\zeta h'(\zeta)}(\Theta(z)-1)\right)\geq \RE \Phi(z)+\RE(\Theta(z)-1)- 5|\Theta(z)-1|.
	\end{equation*}
	From~\eqref{9} and~\eqref{hypo2}, we conclude that $\RE \Psi(z)>0$ and~\eqref{hypo} holds. Therefore the result follows by Lemma~\ref{mainn}.
\end{proof}
\begin{corollary}
	Let $\delta\in[0,1],$ $h\in\mathcal{Q}$ with $0\in\overline{h(\mathbb{D})}$ and $\Theta\in\mathcal{H}$ be a bounded function such that $\Theta(0)=1$ and $|\Theta(z)|\leq M\;(z\in\mathbb{D})$ for some $M>0.$ Suppose $\Phi\in\mathcal{H}$ be such that
	\begin{equation}\label{hypo3}
		\RE \Phi(z)\geq 6(M+1).
	\end{equation}
	If $p\in\mathcal{H}(\delta;\Theta,\Phi),$ $p(0)=h(0)$ and $H_{\delta;\Theta,\Phi,p}\prec h,$ then $p\prec h.$
\end{corollary}
\begin{proof}
	We know that $-\RE(\Theta(z)-1)\leq |\Theta(z)-1|,$ which gives
	\begin{equation*}
		5|\Theta(z)-1|-\RE(\Theta(z)-1)\leq 6|\Theta(z)-1|\leq 6(|\Theta(z)|+1)\leq 6(M+1).
	\end{equation*}
	Clearly~\eqref{hypo3} is sufficient for~\eqref{hypo2} to hold true. Thus the result follows as an application of Theorem~\ref{kcorr}.
\end{proof}
\begin{theorem}
	Let $\delta\in[0,1],$ and $\Theta,\Phi\in\mathcal{H}$ be such that $\Theta(0)=1$ and
	\begin{equation}\label{z1}
		\RE \Phi(z)>2(|\Theta(z)-1|-\RE (\Theta(z)-1)),\quad z\in\mathbb{D}.
	\end{equation}
	If $p\in\mathcal{H}(\delta;\Theta,\Phi),\;p(0)=1$ and $H_{\delta;\Theta,\Phi,p}\prec \sqrt{1+z},$ then $p\prec \sqrt{1+z}.$
\end{theorem}
\begin{proof}
	In view of Lemma~\ref{mainn}, if we take $h(z)=\sqrt{1+z},$ then~\eqref{hypo} reduces to
	\begin{equation*}
		\RE\left(2(\Theta(z)-1)\left(1+\dfrac{1}{\zeta}\right)+\Phi(z)\right)\geq 2\RE(\Theta(z)-1)-2|\Theta(z)-1|+\RE \Phi(z),
	\end{equation*}
	which by~\eqref{z1} is greater than 0. Thus the result follows as an application of Lemma~\ref{mainn}.
\end{proof}
\begin{theorem}
	Let $\delta\in[0,1],$ $\gamma\in(0,1]$ and $\Theta,\Phi\in\mathcal{H}$ be such that $\Theta(0)=1,\Theta'(0)>0$ and $\RE \Phi(z)>0.$ Suppose $p\in\mathcal{H}(\delta;\Theta,\Phi
	),\;p(0)=1$ and $p'(0)>0,$ then
	\begin{equation}\label{z2}
		H_{\delta;\Theta,\Phi,p}\prec\left(\dfrac{1+z}{1-z}\right)^{\gamma}\quad\Rightarrow\quad p(z)\prec \left(\dfrac{1+z}{1-z}\right)^{\gamma}.
	\end{equation}
\end{theorem}
\begin{proof}
	Let $h(z)=((1+z)/(1-z))^{\gamma}.$ From the proof of~\cite[Corollary 6]{pri4}, it is easy to observe that
	\begin{equation*}
		\RE (\Theta(z)-1)\dfrac{h(\zeta)}{\zeta h'(\zeta)}=\RE (\Theta(z)-1)\dfrac{1-\zeta^2}{2\gamma\zeta}>0,
	\end{equation*}
	provided $\Theta'(0)>0$ and $p'(0)>0$ which is true by hypothesis. Using the fact that $\RE \Phi(z)>0$ and Lemma~\ref{mainn}, the implication~\eqref{z2} follows and hence the result.
\end{proof}
\begin{theorem}
	Let $f\in\mathcal{A}$ and $g\in\mathcal{S}^*$ with $g(z)\neq zf'(z)$ and
	\begin{equation}\label{0i}
		\RE\left(\dfrac{2zf'(z)}{g(z)}-\dfrac{2z(f'(z))^2}{3g(z)f'(z)+zf''(z)g(z)-zg'(z)f'(z)}\right)>0,
	\end{equation}
	then $f\in\mathcal{K}.$
	\end{theorem}
	\begin{proof}
		Let $p(z)=zf'(z)/g(z),$ then
		\begin{equation*}
			\dfrac{2zf'(z)}{g(z)}-\dfrac{2z(f'(z))^2}{3g(z)f'(z)+zf''(z)g(z)-zg'(z)f'(z)}=\dfrac{2p(z)(p(z)+zp'(z))}{2p(z)+zp'(z)}.
		\end{equation*}
		From~\eqref{0i}, we have
		\begin{equation*}
			\RE\left(\dfrac{2p(z)(p(z)+zp'(z))}{2p(z)+zp'(z)}\right)>0,
		\end{equation*}
		or equivalently,
		\begin{equation*}
			\dfrac{2p(z)(p(z)+zp'(z))}{2p(z)+zp'(z)}\prec\dfrac{1+z}{1-z}.
		\end{equation*}
		By applying Lemma~\ref{mainn} with $\Theta(z)=\Phi(z)=1,\;t=1/2$ and $h(z)=(1+z)/(1-z),$ we get $p(z)\prec(1+z)/(1-z),$ which is equivalent to
		\begin{equation*}
			\RE p(z)=\RE\dfrac{zf'(z)}{g(z)}>0.
		\end{equation*}
		Hence the result.
	\end{proof}
By taking $p(z)=zf'(z)/f(z)$ in Theorem~\ref{genels}, we obtain the following result.
\begin{corollary}
	Let $\gamma\in[0,1],$ $\alpha\in[0,1),$ $\mu\in[0,1],$ $\delta\in[1,2],$ $\beta\in[0,1)$ and $\rho\in[0,1]$ be such that $\rho\geq \alpha(1+2\alpha),$ whenever $\alpha\in[0,1/2].$ Also, let $f\in\mathcal{A}$ satisfies
	\begin{equation*}
	\RE\left(\gamma\left(\dfrac{zf'(z)}{f(z)}\right)^{\delta}+(1-\gamma)\dfrac{\left(\dfrac{zf'(z)}{f(z)}\right)^{1+\mu}\left(1+\dfrac{zf''(z)}{f'(z)}\right)^{1-\mu}}{\rho\left(1+\dfrac{zf''(z)}{f'(z)}\right)+(1-\rho)\dfrac{zf'(z)}{f(z)}}\right)>\beta,
	\end{equation*}
where $\beta$ is as defined in Theorem~\ref{genels}. Then $f$ is starlike of order $\alpha.$
\end{corollary}
\begin{remark}
	If we take $\delta=1,\;\mu=0$ and $\rho=0$ in the above result, then it implies that the class of $\gamma$-convex functions of order $\beta$ is contained in the class of starlike functions of order $\alpha.$ Further by taking $\alpha=0,$ it reduces to a well known result which states that every $\gamma$-convex functions is starlike(see~\cite{mocanu1969,mocanu1972,mocanu1973}).
\end{remark}
\begin{remark}
By taking $\delta=1,\rho=0$ and $\alpha=0$ in the above result, we obtain a well known result which states that a $\mu$-starlike function is starlike~\cite{lewa,lewa2}.
\end{remark}
By taking $p(z)=f'(z)$ in Theorem~\ref{genels}, we obtain the following result.
\begin{corollary}
	Let $\gamma\in[0,1],$ $\alpha\in[0,1),$ $\mu\in[0,1],$ $\delta\in[1,2],$ $\beta\in[0,1)$ and $\rho\in[0,1]$ be such that $\rho\geq \alpha(1+2\alpha),$ whenever $\alpha\in[0,1/2].$ Also, let $f\in\mathcal{A}$ satisfies
	\begin{equation*}
	\RE\left(\gamma(f'(z))^{\delta}+(1-\gamma)\dfrac{(f'(z))^{1+\mu}\left(f'(z)+\dfrac{zf''(z)}{f'(z)}\right)^{1-\mu}}{f'(z)+\rho\dfrac{zf''(z)}{f'(z)}}\right)>\beta,
	\end{equation*}
	where $\beta$ is as defined in Theorem~\ref{genels}. Then $\RE(f'(z))>\alpha$ and therefore $f$ is univalent in $\mathbb{D}$.
\end{corollary}
By taking $p(z)=f(z)/z$ in Theorem~\ref{genels}, we obtain the following result.
\begin{corollary}
		Let $\gamma\in[0,1],$ $\alpha\in[0,1),$ $\mu\in[0,1],$ $\delta\in[1,2],$ $\beta\in[0,1)$ and $\rho\in[0,1]$ be such that $\rho\geq \alpha(1+2\alpha),$ whenever $\alpha\in[0,1/2].$ Also, let $f\in\mathcal{A}$ satisfies
	\begin{equation*}
		\RE\left(\gamma\left(\dfrac{f(z)}{z}\right)^{\delta}+(1-\gamma)\dfrac{\left(\dfrac{f(z)}{z}\right)^{1+\mu}\left(\dfrac{f(z)}{z}+\dfrac{zf'(z)}{f(z)}-1\right)^{1-\mu}}{\dfrac{f(z)}{z}+\rho\dfrac{zf'(z)}{f(z)}}-1\right)>\beta,
	\end{equation*}
	where $\beta$ is as defined in Theorem~\ref{genels}. Then $\RE(f(z)/z)>\alpha.$
\end{corollary}
%

\end{document}